\documentclass[12pt]{article}

\usepackage{amsmath}
\usepackage{amsfonts}
\usepackage{color}
\usepackage{amssymb}
\usepackage{graphicx}
\usepackage{enumerate}
\usepackage{amsthm,amscd}
\usepackage[all]{xy}

 \allowdisplaybreaks

\usepackage{hyperref}

\newtheorem{theorem}{Theorem}[section]

\newtheorem{conjecture}[theorem]{Conjecture}

\newtheorem{lemma}[theorem]{Lemma}

\begin{document}

\title{The spectral radius of graphs with no 
intersecting odd cycles\thanks{This paper was firstly announced in May, 2021, 
and was later published on Discrete Mathematics 345 (2022) 112907.  
See \url{https://doi.org/10.1016/j.disc.2022.112907}.  E-mail addresses: ytli0921@hnu.edu.cn (Y. Li), 
ypeng1@hnu.edu.cn (Y. Peng, corresponding author).}}

\author{
Yongtao Li, Yuejian Peng$^{\dag}$ \\[2ex]
{\small School of Mathematics, Hunan University} \\
{\small Changsha, Hunan, 410082, P.R. China }  }

\date{January 14, 2022}

\maketitle

\vspace{-0.5cm}

\begin{abstract}
Let $H_{s,t_1,\ldots ,t_k}$ be the graph with   
$s$ triangles and $k$ odd cycles of lengths $t_1,\ldots ,t_k\ge 5$
 intersecting in exactly one common vertex. 
Recently, Hou, Qiu and Liu [Discrete Math. 341 (2018) 126--137],  
and Yuan [J. Graph Theory 89 (1) (2018) 26--39] 
determined independently  
 the maximum number of edges in an $n$-vertex 
 graph that does not contain $H_{s,t_1,\ldots ,t_k}$ as a subgraph. 
In this paper, we  determine the  graphs of order $n$   
that attain the maximum spectral radius 
among all graphs containing no $H_{s,t_1,\ldots ,t_k}$  
for $n$ large enough. 
 \end{abstract}



\section{Introduction}
In this paper, we consider only simple and undirected graphs. Let $G$ be a simple connected  graph with vertex set $V(G)=\{v_1, \ldots, v_n\}$ and edge set $E(G)=\{e_1, \ldots, e_m\}$.  
Let $N(v)$ or $N_G(v)$ be the set of neighbors of $v$, 
and $d(v)$ or $d_G(v)$ be the degree of a vertex $v$ in $G$.
Let $S$ be a set of vertices. We write $d_S(v)$ 
for the number of neighbors of $v$ in the set $S$, 
that is, $d_S(v)=|N(v)\cap S|$. 
And we denote by $e(S)$ the number of edges contained in $S$. 
Let $K_{s,t}$ be the complete bipartite graph with parts of sizes 
$s$ and $t$. In particular, when $s=1$, the graph $K_{1,t}$ 
is the star graph with $t$ edges. 
Let $M_{t}$ be the matching of size $t$, i.e., the union of 
$t$ disjoint edges.

 The {\em Tur\'an number} of a graph $F$ is the maximum number of edges 
  in an $n$-vertex graph without a subgraph isomorphic to $F$, and 
  it is usually  denoted by $\mathrm{ex}(n, F)$. 
  We say that a graph $G$ is $F$-free if it does not contain 
  an isomorphic copy of $F$ as a subgraph. 
  A graph on $n$ vertices with no subgraph $F$ and with $\mathrm{ex}(n, F)$ edges is called an {\em extremal graph} for $F$ and we denote by $\mathrm{Ex}(n, F)$ the set of all extremal graphs on $n$ vertices for $F$. 
  It is  a cornerstone of extremal graph theory 
  to 
  understand $\mathrm{ex}(n, F)$ and $\mathrm{Ex}(n, F)$ for various graphs $F$; 
  see \cite{FS13, Keevash11,Sim13} for surveys.

  In 1941, Tur\'{a}n \cite{Turan41} posed the  natural question of determining
 $\mathrm{ex}(n,K_{r+1})$ for $r\ge 2$. 
 Let $T_r(n)$ denote the complete $r$-partite graph on $n$ vertices where 
 its part sizes are as equal as possible. 
Tur\'{a}n \cite{Turan41} 
(also see \cite[p. 294]{Bollobas78})  extended a result of Mantel \cite{Man07} 
and  obtained that if $G$ is an $n$-vertex graph containing no $K_{r+1}$, 
then $e(G)\le e(T_r(n))$, equality holds if and only if $G=T_r(n)$.  
There are many extensions and generalizations on Tur\'{a}n's result.  
The problem of determining $\mathrm{ex}(n, F)$ is usually called the 
Tur\'{a}n-type extremal problem. 
 The most celebrated extension always attributes to a result of 
 Erd\H{o}s, Stone and Simonovits \cite{ES46,ES66}, which  states that 
 \begin{equation} \label{eqESS}
  \mathrm{ex}(n,F) = \left( 1- \frac{1}{\chi (F) -1} + o(1) \right) 
  \frac{n^2}{2}, 
  \end{equation}
 where $\chi (F)$ is the vertex-chromatic number of $F$. 
 This provides good asymptotic estimates for the extremal numbers of non-bipartite graphs. 
 However, for bipartite graphs, where $\chi (F)=2$, it only gives the bound 
 $\mathrm{ex}(n,F)=o(n^2)$. 
 Although there have been numerous attempts on finding better bounds 
 of $\mathrm{ex}(n,F)$ for various bipartite graphs $F$, we know very little in this case. 
 The history of such a case began in 1954 with 
 the theorem of K\H{o}vari, S\'{o}s and Tur\'{a}n  \cite{KST54}, which states 
 that if $K_{s,t}$ is the complete bipartite graph with vertex classes of size $s\ge t$, 
 then $\mathrm{ex}(n,K_{s,t})=O(n^{2-1/t})$; see \cite{Furedi96,Furedi96b} 
 for more details.  
In particular, we refer the interested reader to 
  the comprehensive survey by F\"{u}redi and Simonovits \cite{FS13}.

\subsection{History and background}
In this section, we shall review the exact values of $\mathrm{ex}(n,F)$ for some 
special graphs $F$, instead of the asymptotic estimation.  
A graph on $2k+1$ vertices consisting of $k$ triangles which intersect in exactly one
common vertex is called a {\it $k$-fan} (also known as the friendship graph) 
and is denoted by $F_k$.
Since $\chi (F_k)=3$, the theorem of Erd\H{o}s, Stone and Simonovits 
 in (\ref{eqESS}) implies that 
$\mathrm{ex}(n,F_k)= n^2/4 + o(n^2)$. 
In 1995, Erd\H{o}s, F\"{u}redi, 
Gould and Gunderson  \cite{Erdos95} proved  the following  exact result.

\begin{theorem}\label{thmErdos95} \cite{Erdos95}
For every $k \geq 1$, and for every $n\geq 50k^2$, 
\[ \mathrm{ex}(n, F_k)= \left\lfloor \frac {n^2}{4}\right \rfloor+ \left\{
  \begin{array}{ll}
   k^2-k, \quad~~  \mbox{if $k$ is odd,} \\
    k^2-\frac32 k, \quad \mbox{if $k$ is even}.
  \end{array}
\right. \]
\end{theorem}

The extremal graphs of Theorem \ref{thmErdos95} are as follows.
For odd $k$ (where $n\geq 4k-1$), 
the extremal graphs are constructed by taking $T_2(n)$,
the balanced complete bipartite  graph, and embedding two vertex
disjoint copies of $K_k$ in one side.
For even $k$ (where now $n\geq 4k-3$), the extremal graphs 
are constructed by taking $T_2(n)$ and embedding
 a graph with $2k-1$ vertices, $k^2-\frac{3}{2} k$ edges with maximum degree $k-1$ in one side. 
 \medskip

 Let $C_{k,q}$ be the graph consisting of $k$ cycles of length $q$ which intersect 
 exactly in one common vertex. Clearly, when we set $q=3$, 
 then $C_{k,3}$ is just the $k$-fan graph; see Theorem \ref{thmErdos95}. 
 When $q$ is an odd integer, we can see that $\chi (C_{k,q})=3$, 
 the theorem of Erd\H{o}s, Stone and Simonovits 
  also implies that 
 $\mathrm{ex}(n,C_{k,q})=n^2/4 + o(n^2)$. 
 In 2016, Hou, Qiu and Liu \cite{HQL16} determined exactly the extremal number 
 for $C_{k,q}$ with $k\ge 1$ and odd integer $q\ge 5$.

 \begin{theorem}\cite{HQL16} \label{thmhql16}
 For an integer $k\ge 1$ and an odd integer $q\ge 5$, 
 there exists $n_0(k,q)$ such that 
 for all $n \ge n_0(k,q)$, we have 
 \[  \mathrm{ex}(n,C_{k,q}) =  \left\lfloor \frac {n^2}{4}\right \rfloor 
 + (k-1)^2. \]
 Moreover, an extremal graph must be a 
 Tur\'{a}n graph $T_2(n)$ with a $K_{k-1,k-1}$ embedding 
 into one class.  
 \end{theorem} 
 
 We remark here that when $q$ is even, then $C_{k,q}$ is a bipartite graph 
 where every vertex in one of its parts
has degree at most $2$. For such a sparse   bipartite graph, 
a classical result of F\"{u}redi \cite{Furedi91} or 
Alon, Krivelevich and Sudakov \cite{Alon03} implies that 
 $\mathrm{ex}(n,C_{k,q})= O(n^{3/2})$. 
 Recently, a breakthrough result of Conlon, Lee and Janzer \cite{CL19,CJL20} shows that 
 for even $q\ge 6$ and $k\ge 1$, we have  
$\mathrm{ex}(n,C_{k,q})= O(n^{3/2-\delta})$ for some $\delta =\delta (k,q)>0$. 
  It is a challenging problem to determine the value $\delta (k,q)$. 
In particular, for the special case $k=1$, this problem 
 reduces to determine the extremal number for even cycle. 

 Next, we shall introduce a unified extension of both Theorem \ref{thmErdos95} 
 and Theorem \ref{thmhql16}. 
 Let $s$ be a positive integer and $t_1,\ldots ,t_k\ge 5$ be odd integers.  
We write $H_{s,t_1,\ldots ,t_k}$ for 
 the graph consisting of $s$ triangles and 
 $k$ odd cycles of  lengths $t_1, \ldots ,t_k $
 in which these triangles and cycles intersect in
exactly one common vertex. 
 The graph $H_{s,t_1,\ldots ,t_k}$ is also known as the flower graph 
with $s+k$ petals. 
We remark here that the $k$ odd cycles 
 can have different lengths. 
 Clearly, when $t_1=\cdots =t_k=0$, then $H_{s,0,\ldots ,0}=F_s$, 
 the $s$-fan graph; see Theorem \ref{thmErdos95}.  
  In addition,  when $s=0$ and  $t_1=\cdots =t_k=q$, 
  then $H_{0,q,\ldots ,q}=C_{k,q}$; 
 see Theorem \ref{thmhql16}. 
 
 In 2018, Hou, Qiu and Liu \cite{HQL18} and Yuan \cite{Yuan18} 
 independently determined the extremal number 
 of $H_{s,t_1,\ldots ,t_k}$ for $s\ge 0 $ and $k\ge 1$.  
 Let $\mathcal{F}_{n,s,k}$ be the family of graphs with each member being a 
 Tur\'{a}n graph $T_2(n)$ with
a graph $Q$ embedded in one partite set, where 
 \[  Q = \begin{cases}
 K_{s+k-1,s+k-1}, & \text{if $(s,k)\neq (3,1)$,} \\
 K_{3,3} ~\text{or}~ 3K_3, & \text{if $(s,k)=(3,1)$,}
 \end{cases}  \]
where $3K_3$ is the union of three disjoint triangles.

 \begin{theorem} \cite{HQL18,Yuan18} \label{thmHY}
  For every graph $H_{s,t_1,\ldots ,t_k}$  with $s\ge 0$ and $ k\ge 1$, 
 there exists $n_0$ such that 
 for all $n \ge n_0$, we have 
 \[  \mathrm{ex}(n,H_{s,t_1,\ldots ,t_k}) =  \left\lfloor \frac {n^2}{4}\right \rfloor 
 + (s+k-1)^2. \]
Moreover, the only extremal graphs for $H_{s,t_1,\ldots ,t_k}$ are members of $\mathcal{F}_{n,s,k}$.
 \end{theorem}

\subsection{Spectral extremal problem}

Let $G$ be a simple graph on $n$ vertices. 
The \emph{adjacency matrix} of $G$ is defined as 
$A(G)=[a_{ij}]_{n \times n}$ where $a_{ij}=1$ if two vertices $v_i$ and $v_j$ are adjacent in $G$, and $a_{ij}=0$ otherwise.   
We say that $G$ has eigenvalues $\lambda_1 ,\ldots ,\lambda_n$ if these values are eigenvalues of 
the adjacency matrix $A(G)$. 
Let $\lambda (G)$ be the maximum  value in absolute 
 among the eigenvalues of $G$, which is 
 known as the {\it spectral radius} of graph $G$, 
 that is, 
 \[  \lambda (G) = \max \{|\lambda | : \text{$\lambda$ 
 is an eigenvalue of $G$}\}. \]
By the Perron--Frobenius Theorem \cite[p. 534]{Horn13}, 
the spectral radius of a graph $G$ is actually 
the largest eigenvalue of $G$ 
since the adjacency matrix $A(G)$ is nonnegative.  
We usually write $\lambda_1(G)$ for the spectral radius of $G$. 
The spectral radius of a graph sometimes can give some information  
about the structure of graphs. 
For example, it is well-known \cite[p. 34]{Bapat14} that the average degree of $G$ is at most $\lambda (G)$, which is at most the maximum degree of $G$. 

In this paper we consider spectral analogues 
of Tur\'{a}n-type problems for graphs. 
That is, determining $\mathrm{ex}_{sp}(n,F)= \max \{\lambda (G) : 
|G|=n, F\nsubseteq G \}$.  It is well-known that 
\begin{equation}
 \mathrm{ex}(n,F) \le \frac{n}{2} \mathrm{ex}_{sp}(n,F) 
 \end{equation}
because of the  fundamental inequality $\frac{2m}{n} \le \lambda (G)$.  
For most graphs, this study is again fairly complete 
due in large part to a longstanding work of Nikiforov \cite{NikifSurvey}. 
For example, he extended the classical theorem of Tur\'{a}n, 
by determining the maximum spectral radius of 
any $K_{r+1}$-free graph $G$ on $n$ vertices.

The following problem regarding the adjacency  spectral radius
was proposed in \cite{NikiforovTuran}:
What is the maximum spectral radius of a graph $G$ on $n$
vertices without a subgraph isomorphic to a given graph $F$?
Wilf \cite{Wilf86} and Nikiforov \cite{NikiforovTuran} 
obtained spectral strengthening of Tur\'an's theorem 
when the forbidden  substructure is the complete graph. 
Soon after, 
Nikiforov \cite{Nikiforov07} showed that if $G$ is a $K_{r+1}$-free graph on $n$ vertices, 
then $\lambda (G)\le \lambda (T_r(n))$, 
equality holds if and only if $G=T_r(n)$. 
Moreover, 
Nikiforov \cite{Nikiforov07} (when $n$ is odd), 
and Zhai and Wang \cite{ZW12} (when $n$ is even) 
determined the maximum spectral radius 
of $K_{2,2}$-free graphs. 
Furthermore, Nikiforov \cite{NikiforovKST}, Babai and Guiduli \cite{BG09} 
independently  obtained  the spectral generalization of the 
theorem of K\H{o}vari, S\'os and Tur\'an  when the forbidden graph is the complete bipartite graph $K_{s,t}$. 
Finally, Nikiforov \cite{NikiforovLAA10} characterized 
the spectral radius of graphs without paths and cycles of specified length. 
In addition, Fiedler and Nikiforov  \cite{FiedlerNikif} obtained tight sufficient conditions for
graphs to be Hamiltonian or traceable. 
For many other spectral analogues of results in extremal graph theory 
we refer the reader to the survey  \cite{NikifSurvey}. 
It is worth mentioning that a corresponding spectral extension \cite{Niki09} of 
the theorem of Erd\H{o}s, Stone and Simonovits  states that 
\[  \mathrm{ex}_{sp}(n,F) = \left( 1- \frac{1}{\chi (F)-1} +o(1) \right)n. \]
From this result, we know that 
$\mathrm{ex}_{sp}(n,F_k) = n/2 +o(n)$ where $F_k$ is the $k$-fan graph. 
Recently,  Cioab\u{a}, Feng,  Tait and  Zhang \cite{CFTZ20}   generalized this bound by 
improving the error term $o(n)$ to $O(1)$, 
and obtained a spectral counterpart of Theorem \ref{thmErdos95}. More precisely, 
they proved the following theorem.

 \begin{theorem} \cite{CFTZ20}   \label{thmCFTZ20}
Let $G$ be a graph of order $n$ that does not contain a copy of $F_k$ where $k \geq 2$.  
For sufficiently large $n$, if $G$ has the maximal spectral radius, then
$$G \in \mathrm{Ex}(n, F_k).$$
\end{theorem}

Recall that $H_{s,t_1,\ldots ,t_k}$ is the graph consisting of $s$ triangles and 
 $k$ odd cycles of lengths $t_1,\ldots ,t_k$ which intersect in
exactly one common vertex. 
In this paper, we shall prove the following theorem. 

 \begin{theorem}[Main result] \label{thmmain}
Let $G$ be a graph of order $n$ that does not contain a copy of $H_{s,t_1,\ldots ,t_k}$, where $s\ge 0$ and $k \geq 1$.  
For sufficiently large $n$, if $G$ has the maximal spectral radius, then
$$G \in \mathrm{Ex}(n, H_{s,t_1,\ldots ,t_k}).$$
\end{theorem}

It is interesting that 
the spectral extremal example sometimes differs from 
the usual extremal example. 
For instance, Nikiforov \cite{Nikiforov07},  
and Zhai and Wang  \cite{ZW12}  proved that 
the maximum spectral radius of a $C_4$-free graph 
on $n$ vertices is uniquely achieved by the friendship 
graph. This is very different from 
the usual extremal problem 
for the maximum number of edges in 
a $C_4$-free graphs, 
since F\"{u}redi \cite{Furedi96c} showed that 
for $n$ large enough with the form 
$n=q^2+q+1$, the extremal number is 
attained by the polarity graph of a projective plane. 
From Theorem \ref{thmCFTZ20} and Theorem \ref{thmmain}, 
we know that  graphs attaining the maximum spectral 
radius among all $F_k$-free ($H_{s,t_1,\ldots ,t_k}$-free) graphs 
also 
contain the maximum number of edges 
among all $F_k$-free ($H_{s,t_1,\ldots ,t_k}$-free) graphs.

Our theorem  is a spectral result of 
the Tur\'{a}n extremal problem for $H_{s,t_1,\ldots ,t_k}$, it 
  can be viewed as 
 an extension of Theorem \ref{thmCFTZ20}, as well as 
a spectral analogue of Theorem \ref{thmHY}.  
Our treatment strategy of the proof is mainly 
based on the stability method. 
To some extent, this paper could be regarded 
as a continuation and development of \cite{CFTZ20}. 
The heart of the proof and all key ideas lie in the proof of 
stability. 
We know that if we forbid the substructure $F_k$, 
then the neighborhood of each vertex does not contain 
a matching of $k$ edges. 
While we forbid the intersecting odd-length cycles, 
the neighborhood of each vertex does not contain a 
long path, which can be viewed as 
a key observation in our extension. 
In addition, the embedding method 
of $H_{s,t_1,\ldots ,t_k}$ 
is slightly different from that of $F_k$, 
we need to prove the existence of a larger bipartite subgraph. 
We remark here that the spectral stability method 
is also used in a recent paper 
to deal with the extremal problem of odd-wheel graph \cite{CDT21}.

\section{Some Lemmas} 

In this section, 
we state some lemmas which are needed in our proof.

\begin{lemma} \cite{EG59}   \label{lempath}
Let $P_t$ denote the path on $t$ vertices. 
If $G$ is a $P_t$-free graph on $n$ vertices, 
then $e(G)\le \frac{(t-2)n}{2}$, equality holds 
if and only if $G$ is the disjoint union of copies of $K_{t-1}$. 
\end{lemma}

\begin{lemma} \cite{CFTZ20} \label{lemlb}
If  $G$ has $t$ triangles, then
$e(G) \geq \lambda(G)^2 - \frac{3t}{\lambda (G)}$. 
\end{lemma}

The next is the famous triangle removal lemma \cite{Ruzsa76}, 
which is a direct consequence of the Szemer\'{e}di regularity lemma; 
see, e.g., \cite{Conlon13, fox2011} for more details.  

\begin{lemma} \cite{Ruzsa76} \label{triangleremoval2}
For every $\varepsilon>0$, there exists  $
\delta (\varepsilon)  > 0$ such that every  $n$-vertex graph with at
most $\delta (\varepsilon) \cdot n^3$  triangles  can be made triangle-free by removing at most $\varepsilon n^2$ edges.
\end{lemma}

\begin{lemma}[F\"uredi \cite{Furedi2015}]  \label{furedi20153}
Let $G$ be a triangle-free graph on $n$ vertices. 
If  $a>0$ and
$e(G) = e(T_{2}(n)) -a$, 
then there exists a bipartite subgraph $H \subseteq G$ such that
$e(H) \ge e(G) - a$. 
\end{lemma}

Let $G$ be a simple graph with matching number $\beta(G)$ and maximum degree $\Delta(G)$. For given two integers $\beta$ and $\Delta$, define $f(\beta, \Delta) 
=\max\{e(G): \beta(G)\leq \beta, \Delta(G)\leq \Delta \}$.

In 1976, Chv\'atal and Hanson \cite{Chvatal76} obtained the following result.

\begin{lemma}[Chv\'atal--Hanson \cite{Chvatal76}]\label{Chvatal76}
For every two integers $\beta \geq 1$ and $\Delta \geq 1$, we have
$$f(\beta, \Delta)= \Delta \beta +\left\lfloor\frac{\Delta}{2}\right\rfloor
 \left \lfloor \frac{\beta}{\lceil{\Delta}/{2}\rceil }\right \rfloor
 \leq \Delta \beta+\beta.$$
\end{lemma}

We will frequently use a special case proved by Abbott, Hanson and Sauer \cite{Abbott72}:
$$f(k-1,k-1) = \left\{
  \begin{array}{ll}
   k^2-k, \quad~~  \mbox{if $k$ is odd,} \\
    k^2-\frac32 k, \quad  \mbox{if $k$ is even}.
  \end{array}
\right.$$
Furthermore, the extremal graphs attaining the equality case 
are exactly those we embedded into the Tur\'{a}n graph $T_{2}(n)$  to obtain the extremal $F_k$-free graph.

\section{Proof of  Theorem~\ref{thmmain}}

In the sequel, 
we always assume that $G$ is
 a graph on $n$ vertices 
 containing no $H_{s,t_1,\ldots ,t_k}$ as a subgraph and attaining 
 the maximum spectral radius. 
 The aim of this section is to prove that 
 $e(G) = \mathrm{ex}(n, H_{s,t_1,\ldots ,t_k})$ for $n$ large enough.

First of all, we note that $G$ must be connected 
since adding an edge between different 
components will increase the spectral radius 
and also keep the resulting graph being $H_{s,t_1,\ldots ,t_k}$-free.
Let $\lambda (G)$ be the spectral radius of $G$. 
By the Perron--Frobenius Theorem \cite[p. 534]{Horn13}, 
we know that $\lambda_1$
has an  eigenvector with all entries  positive. 
We denote such an eigenvector by  $\mathbf{x}$.  
For a vertex $v\in V(G)$, 
we will write $\mathbf{x}_v$ for 
the eigenvector entry of $\mathbf{x}$ corresponding to $v$. 
We  may normalize $\mathbf{x}$ so that it has maximum entry equal to  $1$, and let $z$ be a vertex such that $\mathbf{x}_z = 1$. 
If there are multiple such vertices, 
we choose and fix $z$ arbitrarily among them. 

In the sequel, 
we shall prove Theorem \ref{thmmain} iteratively, giving successively better lower bounds on both $e(G)$ and the eigenvector entries of all of the other vertices, until finally we can show that $e(G) = \mathrm{ex}(n, H_{s,t_1,\ldots ,t_k})$.

The proof of Theorem~\ref{thmmain}  is outlined as follows.

\begin{itemize}
\item[$\spadesuit$] 
We apply Lemma \ref{lemlb} to give a lower bound 
 $e(G)\ge \frac{n^2}{4} - O(n)$; see Lemma \ref{lemma31}. 
 Then we use the triangle removal lemma and F\"uredi's stability result, and show that $G$ has a very large bipartite subgraph 
on parts $S,T$ with $\frac{n}{2} - o(n)\le |S|,|T| 
\le \frac{n}{2}+ o(n)$. 
Moreover, we also have $e(S,T)\ge \frac{n^2}{4} - o(n^2)$; 
see Lemma \ref{maxcut}.

\item[$\heartsuit$] 
We show that 
the number of vertices 
that have $\Omega (n)$ neighbors on its side of the partition 
is bounded by $o(n)$,  and the number of vertices 
that have degree less than $(\frac{1}{2}- c_0 )n$ is bounded by 
$O(1)$ for some small constant $c_0>0$; see Lemmas \ref{Wupper} and \ref{Lupper} respectively. 
Furthermore, 
we will prove that 
such vertices do not exist, 
and both $G[S]$ and $G[T]$ are $K_{1,s+k}$-free and $M_{s+k}$-free; 
see Lemmas \ref{lem355}, \ref{lem336} and \ref{Lempty}.

\item[$\clubsuit$] 
Based on the previous lemmas, 
we shall refine the structure of $G$, and 
improve the lower bound of $e(G)$ to 
$e(G)\ge \frac{n^2}{4}- O(1)$ and 
refine the bisection $\frac{n}{2} - O(1)\le |S|,|T| \le \frac{n}{2} +O(1)$ 
and also $e(S,T) \ge \frac{n^2}{4} - O(1)$; 
see Lemma \ref{STlambdarefine}. 
Moreover, we shall prove that 
$\mathbf{x}_u = 1- o(1)$ for every $u\in V(G)$; 
see Lemma \ref{second lower bound evector2}.

\item[$\diamondsuit$] 
Once we know that all vertices have eigenvector entries  close to $1$, we can show that the bipartition is balanced; 
see Lemmas \ref{lemma311}, \ref{lemma312} and \ref{cut balanced}. 
This implies that $G$ can be converted to a graph in 
$\mathrm{Ex}(n, H_{s,t_1,\ldots ,t_k})$ by  deleting small number of 
edges within $S,T$ and adding small number of edges 
between $S$ and $T$.   
Invoking these facts, 
we finally show that $e(G) = \mathrm{ex}(n, H_{s,t_1,\ldots ,t_k})$.
\end{itemize}

 Let $H $ 
 be an $H_{s,t_1,\ldots ,t_k}$-free  
 graph on $n$ vertices with maximum number of edges. 
 Since $G$ is the graph maximizing the spectral radius over all $H_{s,t_1,\ldots ,t_k}$-free graphs, in view of Theorem \ref{thmHY},
 we can see by the Rayleigh quotient \cite[p. 234]{Horn13} 
 or \cite[p. 267]{Zhang11} that 
\begin{equation}\label{first lower bound1}
\lambda (G) \geq \lambda (H) \geq \frac{\mathbf{1}^T A(H) \mathbf{1}}{\mathbf{1}^T\mathbf{1}}
= \frac{ 2(\left\lfloor {n^2}/{4}\right \rfloor 
 + (s+k-1)^2)}{n} >\frac{n}{2}.
\end{equation}

\begin{lemma} \label{lemma31}
Let $c$ be the largest length of  cycles of $H_{s,t_1,\ldots ,t_k}$. 
Then 
  \begin{equation}\label{maxsize}
e(G) \geq  \frac{n^2}{4}-(s+k)cn.
  \end{equation}
\end{lemma}

\begin{proof}
Since $G$ is $H_{s,t_1,\ldots ,t_k}$-free, 
the neighborhood of any vertex does not contain $P_{(s+k)c}$ (a path on $(s+k)c$ vertices) as a subgraph. 
Otherwise, $G$ contains the join graph 
$K_1 \vee P_{(s+k)c}$, which contains a copy of $H_{s,t_1,\ldots ,t_k}$. 
Let $t$ be the number of triangles in $G$.
By Lemma~\ref{lempath}, we can obtain the following upper bound on $t$, 
\[3t = \sum_{v\in V(G)} e(G[N(v)]) \leq  \sum_{v\in V(G)} \mathrm{ex}(n, P_{(s+k)c}) < \sum_{v\in V(G)}  
\frac{(s+k)cn}{2}=\frac{(s+k)c}{2}n^2.\] 
This gives $t \leq \frac{(s+k)c}{6}n^2 $. 
From Lemma \ref{lemlb} and (\ref{first lower bound1}), we obtain
  \begin{equation}\label{maxsize}
e(G) \geq  \lambda^2(G) -\frac{6t}{n} \geq \frac{n^2}{4}-(s+k)cn.
  \end{equation}
  This completes the proof. 
\end{proof}

\begin{lemma}\label{maxcut}
Let $\varepsilon$  be a fixed positive constant. There exists an $n_0(\varepsilon, s,k,c)$ such that $G$ has a partition $V=S\cup T$ which gives a  maximum bipartite subgraph,  and
$$e(S, T)\ge \left  (\frac{1}{4}-\varepsilon \right)n^2$$
for $n\ge n_0(\varepsilon, s,k,c)$. Furthermore
\begin{equation} \label{eqeq2}
\left(\frac{1}{2} - \sqrt{\varepsilon} \right) n \leq |S|, |T| \leq \left(\frac{1}{2} + \sqrt{\varepsilon}\right) n.
\end{equation}
\end{lemma}

\begin{proof}
  Let $\delta (\frac{\varepsilon }{4})$ be the parameter 
  chosen from  the Triangle Removal Lemma  \ref{triangleremoval2}. 
In the proof of Lemma \ref{lemma31}, 
we know that  $t \leq  \frac{(s+k)c}{6n}n^3\le 
\delta (\frac{\varepsilon}{4}) n^3$ for $n\ge  n_0= \frac{(s+k)c}{6\delta (\varepsilon / 4)}$.   By Lemma \ref{triangleremoval2}, there exists a large integer $n_0$ such that  the graph $G_1$ obtained from $G$ by deleting
       at most $\frac{\varepsilon}{4} n^2$ edges is $K_3$-free for  $n\ge n_0$. Hence the size of the   graph $G_1$ of order $n$ satisfies
    \[ e(G_1)\ge e(G)-\frac{\varepsilon}{4} n^2\ge 
 \frac{n^2}{4}  - 
    (s+k)cn -\frac{\varepsilon}{4}n^2.\]
      Note that  $e(G_1)\le   e(T_{2}(n))$ by the Mantel Theorem.  
      We define
$a:= e(T_{2}(n))-e(G_1)$, then 
$0\le a \le (s+k)cn +\frac{\varepsilon}{4}n^2$. 
    By Lemma \ref{furedi20153}, $G_1$ contains a bipartite subgraph $G_2$ such that $e(G_2)\ge e(G_1)- a$.
    Hence,    for $n$ sufficiently large, we have
    \begin{eqnarray*}
    e(G_2) \ge  e(G_1)- a 
    \ge  \frac{n^2}{4}  - 
    2(s+k)cn -\frac{\varepsilon}{2}n^2
 \ge        \left (\frac{1}{4}-\varepsilon \right) n^2.
    \end{eqnarray*}
        Therefore,   $G$ has  a partition  $V=S\cup T$ which gives a  maximum cut  such that
     \begin{equation}\label{maxcut1}
     e(S,T)\ge e(G_2)\ge \left  (\frac{1}{4}-\varepsilon \right) n^2.
    \end{equation}
    Furthermore, without loss of generality, we may assume that  $|S|\le |T|$. If $|S|<(\frac{1}{2}-\sqrt{\varepsilon})n$, then $|T|=n-|S|>
    (\frac{1}{2}+\sqrt{\varepsilon})n$. So
    $$e(S,T)\le |S||T| <  \left  (\frac{1}{2}-\sqrt{\varepsilon} \right)n  \left (\frac{1}{2}+\sqrt{\varepsilon} \right)n=\left (\frac{1}{4}-\varepsilon \right)n^2,$$
    which contradicts to  Eq. (\ref{maxcut1}).
    Therefore it follows that
    $$\left (\frac{1}{2}-\sqrt{\varepsilon} \right)n\le |S|, |T|\le \left  (\frac{1}{2}+\sqrt{\varepsilon} \right)n.$$
    Hence the assertion (\ref{eqeq2}) holds.
        \end{proof}

For a vertex $v$, let $d_S(v) = |N(v) \cap S|$ and $d_T(v) = |N(v) \cap T|$. 
Next, we consider the set of vertices that have many neighbors which are not in the cut.

\begin{lemma}\label{Wupper}
Let $\varepsilon , \delta $ be two sufficiently small constants with 
$\varepsilon <{\delta^2}/{3}$. 
We denote
\begin{equation} \label{defW}
W: = \left\{ v\in S: d_S(v) \geq \delta n\right\} \cup \left\{v \in T: d_T(v) \geq \delta n\right\}. 
\end{equation}
For  sufficiently large $n$, we have
\[ |W| \le \frac{2\delta }{3} n + 
 \frac{ 2(s+k-1)^2}{\delta n} <\delta n.  \]
\end{lemma} 

\begin{proof}
Firstly, by Theorem \ref{thmHY}, we know that 
$e(G)\le \mathrm{ex} (n, H_{s,t_1,\ldots ,t_k})\le \frac{n^2}{4}+(s+k-1)^2$.  
   Note that $e(S,T) \geq \left(\frac{1}{4} - \varepsilon \right)n^2$ by  Lemma~\ref{maxcut}.
 Hence
 \begin{equation}
 \begin{aligned}\label{WL-1}
 e(S)+e(T)=e(G)-e(S, T) &\le \frac{n^2}{4}+
 (s+k-1)^2- \left(\tfrac{1}{4} - \varepsilon \right)n^2 \\ 
 & =\varepsilon n^2+(s+k-1)^2.
 \end{aligned}
 \end{equation}
 On the other hand, if we
 denote  $W_1=W\cap S$ and $W_2=W\cap T$, then we get
 \[ 2e(S) =\sum_{u\in S}d_{S}(u) \ge  \sum_{u\in W_1}d_S(u)\ge |W_1|\delta n, \] 
 and similarly, we also have 
 \[   2e(T) = \sum_{u\in T}d_{T}(u)\ge  \sum_{u\in W_2}d_T(u)\ge |W_2|\delta n.\]
  So
  \begin{equation}\label{WL-2}
  e(S)+e(T)\ge (|W_1|+|W_2|)\frac{\delta n}{2} 
  =\frac{\delta n}{2} |W|.
  \end{equation}
  Combining \eqref{WL-1} and \eqref{WL-2}, we get
$ \frac{\delta n}{2}|W|\le \varepsilon n^2+(s+k-1)^2$, 
  i.e.,
\[ |W|\le \frac{2\varepsilon n^2+ 2(s+k-1)^2}{\delta n}. \] 
Note that $\varepsilon < \delta^2/3$, 
we can get $|W|< \delta n$ for sufficiently large $n$. 
\end{proof} 

\begin{lemma}\label{Lupper} 
Let $k\geq 2$.
We denote $c_0:=\frac{1}{8c(s+k)}$  and 
\begin{equation} \label{defL}
L:= \Bigl\{v\in V(G): d(v) \leq 
\left(\tfrac{1}{2}- c_0 \right) n\Bigr\}.
\end{equation}
Then
$$|L|\le 16c^2(s+k)^2. $$
\end{lemma}
\begin{proof}
 Suppose that $|L| >16c^2(s+k)^2$.
 Then let $L^{\prime}\subseteq L$ with $|L^{\prime}|=16c^2(s+k)^2$.
 Then  it follows that
   \begin{eqnarray*}
e(G-L^{\prime}) &\ge& e(G)-\sum_{v\in L^{\prime}}d(v)\\
&\ge& \frac{n^2}{4}-(s+k)cn- 
16c^2(s+k)^2 \left(\frac{1}{2}-\frac{1}{8c(s+k)} \right)n\\
&=& \frac{n^2}{4}  - 8c^2(s+k)^2n +(s+k)cn \\ 
&>& \frac{(n-16c^2(s+k)^2)^2}{4}+(s+k-1)^2 
     \end{eqnarray*}
 for  sufficiently large $n$, where the second inequality is by \eqref{maxsize}. Hence by Theorem~\ref{thmHY}, $G-L^{\prime}$ contains $H_{s,t_1,\ldots ,t_k}$, which implies that $G$ contains $H_{s,t_1,\ldots ,t_k}$.
 So the assertion holds.
 \end{proof}

Now, we have proved that 
 $|W|=o(n)$ and $|L| = O(1)$ by Lemmas   \ref{Wupper} 
 and \ref{Lupper}, respectively. 
Next we will improve the bound on $W$ and 
actually show that 
$W$ is  a subset of $L$, so $|W|=O(1)$.  
To proceed, 
we first  need the following lemma which can be proved by induction or double counting.

\begin{lemma}\label{set}
Let $A_1,A_2, \ldots, A_p$ be $p$  finite sets. Then
\begin{equation*} 
\left| \, \bigcap_{i=1}^p A_i \, \right|\ge \sum_{i=1}^p|A_i|-(p-1)\left| \, \bigcup_{i=1}^pA_i \, \right|.
\end{equation*}
\end{lemma}

\begin{lemma} \label{lem355}
Let $W$ and $L$ be sets of vertices defined in  
 (\ref{defW}) and (\ref{defL}). Then 
$W \subseteq L$. 
\end{lemma}

\begin{proof}
 Suppose on the contrary that there exists a vertex $u_0 \in W$ and $u_0 \notin L$. 
 Recall that $W_1=W\cap S$ and $W_2=W\cap T$. Similarly, 
   let $L_1=L\cap S$ and $L_2=L\cap T$. Without loss of generality, we may assume that $u_0\in S$, 
   that is, $u_0\in W_1$ and $u_0\notin  L_1.$ Since  $S$ and $T$ form a maximum bipartite subgraph, 
   we have $d_T(u_0)\ge \frac{1}{2}d(u_0)$. Indeed, otherwise, 
   we can move the vertex $u_0$ into the part $T$, it will increase 
 strictly  the number of edges between $S$ and $T$. 
   On the other hand, 
   invoking the fact  $u_0 \not\in L$, we get $d(u_0)\ge 
   (\frac{1}{2}-\frac{1}{8c(s+k)} )n$. So
   \[ d_T(u_0)\ge \frac{1}{2}d(u_0)\ge \left(\frac{1}{4}-\frac{1}{16c(s+k)} \right)n.\]
Recall in Lemmas   \ref{Wupper} and \ref{Lupper} that 
\[  |W| <{\delta} n, \quad 
|L|\le 16c^2(s+k)^2.\]
  Hence, for fixed $\delta<\frac{1}{10(k+1)^2}$ and sufficiently large $n$, we have
\begin{equation} \label{eqlarge}
  |S\setminus (W\cup L)|\ge 
  \left(\frac{1}{2}- \sqrt{\varepsilon } \right)n- 
 \delta n - 
  16c^2(s+k)^2\ge  (s+k)c.
\end{equation}

{\bf Claim.} $u_0$ is adjacent to at most $s+k-1$ vertices in  $S\setminus (W\cup L)$.  

\medskip 

   Suppose that $u_0$ is adjacent to   $s+k$ vertices  $u_1, u_2, \ldots, u_{s+k}$ in $S\setminus (W\cup L)$.  Since $u_i\not\in L$,  we have $d(u_i)\ge (\frac{1}{2}-\frac{1}{8c(s+k)} )n$. On the other hand, 
   we have $d_S(u_i)\le \delta n$ because  $u_i\notin W$. 
   So $d_T(u_i)=d(u_i)-d_S(u_i)\ge 
  (\frac{1}{2}-\frac{1}{8c(s+k)}-\delta )n $. 
  In addition, 
  we can choose other vertices $u_{s+k+1}, \ldots ,u_{(s+k)c}$ in 
  the set $S \setminus (W \cup L)$. Similarly, 
  we also have $d_S(u_i) \ge 
  (\frac{1}{2} - \frac{1}{8c(s+k)} -\delta )n$ 
  for each $i\in [s+k+1,(s+k)c]$. 
   By Lemma~\ref{set}, we consider the cardinality of common neighbors 
   \begin{eqnarray*}
&& \left| N_T(u_0) \cap N_T(u_1) 
\cap  \cdots \cap N_T(u_{(s+k)c}) \right| \\ 
 &\ge&
    \sum_{i=0}^{(s+k)c} \left| N_T(u_i) \right| 
    - (s+k)c \left| \bigcup_{i=0}^{(s+k)c} N_T(u_i) \right|\\
 &\ge&  d_T(u_0)+ d_T(u_1)+\cdots +d_T(u_{(s+k)c}) - 
 (s+k)c|T| \\
& \ge &\left (\frac{1}{4}-\frac{1}{16c(s+k)}\right)n+ 
\left(\frac{1}{2}-\frac{1}{8c(s+k)}-\delta\right)n\cdot (s+k)c - 
(s+k)c \left(\frac{1}{2}+\sqrt{\varepsilon}\right)n  \\
 &= &\left(\frac{1}{8}- \frac{1}{16c(s+k)}- 
 (s+k)c\delta- (s+k)c \sqrt{\varepsilon}\right)n> (s+k)c
\end{eqnarray*}
for sufficiently large $n$, where the last inequality follows from 
the fact that $\delta $ and $\varepsilon$ are small enough, 
e.g., $\delta<\frac{1}{100c^2(s+k)^2}$ and $\varepsilon<\frac{\delta^2}{3}$. So there exist $(s+k)c$  vertices $v_1,v_2, \ldots v_{(s+k)c}$  in $T$ such that the induced  subgraph by two partitions $\{u_1, \ldots, u_{(s+k)c}\}$ and $\{v_1, \ldots, v_{(s+k)c}\}$ is complete bipartite. 
The subgraph of $G$ formed by the vertex $u_0$ 
together with such a complete bipartite graph 
can contain many disjoint odd-length cycles. 
For example, we can choose $u_0u_1v_1u_0$ to 
find a copy of triangle, 
and we can choose $u_0u_1v_1u_{s+k+1}v_2u_0$ 
to form a copy of pentagon and so on. 
Hence, it follows that $G$ contains $H_{s,t_1,\ldots ,t_k}$,
this is a contradiction.
Therefore $u_0$ is adjacent to at most $s+k-1$ vertices in  $S\setminus (W\cup L)$.

Hence, applying Lemmas   \ref{Wupper} and \ref{Lupper} again, 
we have
   \begin{eqnarray*}
   d_S(u_0)&\le & |W|+|L|+s+k-1\\
   &<&   \frac{2\delta }{3} n + 
 \frac{ 2(s+k-1)^2}{\delta n}  +16c^2(s+k)^2+ s+k-1\\ 
&<& \delta n
\end{eqnarray*}
for sufficiently large $n$.  This is a contradiction to the fact that $u_0\in W$.
Similarly, there  is no vertex  $u$ such that 
$u\in W_2$ and $u\notin L_2$.  
Hence $W\subseteq L$. 
 \end{proof}

 \begin{lemma} \label{lem336}
There exist  independent sets
 $I_S\subseteq S$  and $I_T\subseteq T$ such that
\[  |I_S|\ge  |S|-20c^2 (s+k)^2 \quad
 \mbox {and} \quad
  |I_T|\ge |T|-20c^2 (s+k)^2. \]
 \end{lemma}

 \begin{proof}
 Since $S\setminus L$ is large enough 
 by reviewing (\ref{eqlarge}) 
 in the proof of Lemma \ref{lem355},  
 we next prove that there exists a large complete bipartite subgraph 
 between $S$ and $T$. 
 Let  $u_1, \ldots, u_{(s+k)c} $ be $(s+k)c$ vertices 
 chosen {\it arbitrarily} from $S \setminus L$. 
 Since $u_i\notin L$, we have 
$$
d(u_i)\ge \left(\frac{1}{2}-\frac{1}{8c(s+k)} \right)n.
$$
Note that $W\subseteq L$ by Lemma~\ref{lem355}, 
so $u_i \notin W$, then $d_S(u_i)\le \delta n$. Hence
\[ d_T(u_i)=d(u_i)-d_S(u_i)\ge
 \left(\frac{1}{2}-\frac{1}{8c(s+k)}-\delta\right)n. \] 
 Furthermore, by Lemma~\ref{set}, we  have
    \begin{eqnarray*}
\left |\bigcap_{i=1}^{(s+k)c}N_T(u_i) \right| 
&\ge &  \sum_{i=1}^{(s+k)c}|N_T(u_i)|-((s+k)c -1) \left|\bigcup_{i=1}^{(s+k)c}N_T(u_i) \right| \\
& \ge& \left (\frac{1}{2}-\frac{1}{8c(s+k)}-\delta \right)n\cdot 
(s+k)c -((s+k)c-1) \left(\frac{1}{2}+\sqrt{\varepsilon} \right)n \\
 &=&\left( \frac{3}{8} - (s+k)c\delta -((s+k)c-1)\sqrt{\varepsilon} \right)n 
 > (s+k)c
\end{eqnarray*}
 for  sufficiently  large $n$.
Hence for any $(s+k)c$ vertices $u_1,u_2, \ldots ,u_{(s+k)c}$ in 
$S\setminus L$, 
there exist $(s+k)c$ vertices $v_1,v_2, \ldots, v_{(s+k)c}\in T$ such that the subgraph formed by two partitions 
$\{u_1, \ldots, u_{(s+k)c}\}$ and 
$\{v_1,\ldots, v_{(s+k)c}\}$ is a complete bipartite graph.  

\medskip 
{\bf Claim.}  {\it $G[S\setminus L]$ is both $K_{1, s+k}$-free 
and $M_{s+k}$-free. }
\medskip

If $G[S\setminus L]$ contains a copy of 
$K_{1,s+k}$ centered at vertex $u_0$ with leaves $u_1,u_2,\ldots ,u_{s+k}$, then 
by the discussion above, there 
exist $u_{s+k+1},\ldots $, $u_{(s+k)c}$ $\in S\setminus (L\cup \{u_1,\ldots ,u_{s+k}\})$ and 
$v_1, \ldots ,v_{(s+k)c}\in T$ such that 
$\{u_1,\ldots ,u_{(s+k)c}\}$ and $\{v_1, \ldots ,v_{(s+k)c}\}$ 
form a complete bipartite graph. 
Therefore, we can see that 
$u_0u_1v_1$, $u_0u_2v_2$, $\ldots, u_0u_sv_s$ form 
$s$ triangles centered at $u_0$. Moreover, there exist 
$u_x,\ldots ,u_z \in S\setminus (L\cup \{u_1,\ldots ,u_{s+k}\})$ and 
$v_y,\ldots ,v_w\in T \setminus \{v_1,\ldots ,v_{s+k}\}$ such that 
  $u_0u_{s+1}v_{s+1} u_{x} v_{y}\cdots u_zv_w 
   u_0$ forms an odd cycle
and in fact we can find all other odd cycles similarly. 
Since there are $(s+k)c$ such $u_i$ and $(s+k)c$ such $v_i$, 
it is enough to form $s$ triangles and $k$ odd cycles 
of lengths no more than $c$. 
Hence there is a copy of $H_{s,t_1,\ldots ,t_k}$ centered at $u_0$. 
Therefore, $G[S\setminus L]$ is $K_{1,s+k}$-free. 
Now, 
we assume that
 $\{u_1u_2,u_3u_4, \ldots ,u_{2(s+k)-1}u_{2(s+k)}\}$ is
  a matching of size $s+k$ in $S\setminus L$. 
  Then $v_1u_1u_2, v_1u_3u_4, \ldots ,v_1u_{2s-1}u_{2s}$ 
  form $s$ triangles centered at vertex $v_1$, and there exist 
  distinct vertices $v_i ,\ldots ,v_x \in T \setminus \{v_1\}$ 
    and $u_j, \ldots ,u_y\in S\setminus (L\cup \{u_1,\ldots ,u_{2s}\})$
 such that $v_1u_{2s+1}u_{2s+2} v_i u_j \cdots v_x u_y v_{1}$ 
  forms an odd cycle and so on. 
  So $G[S\setminus L]$ is $M_{s+k}$-free.  
  \begin{center} 
\includegraphics[scale=0.085]{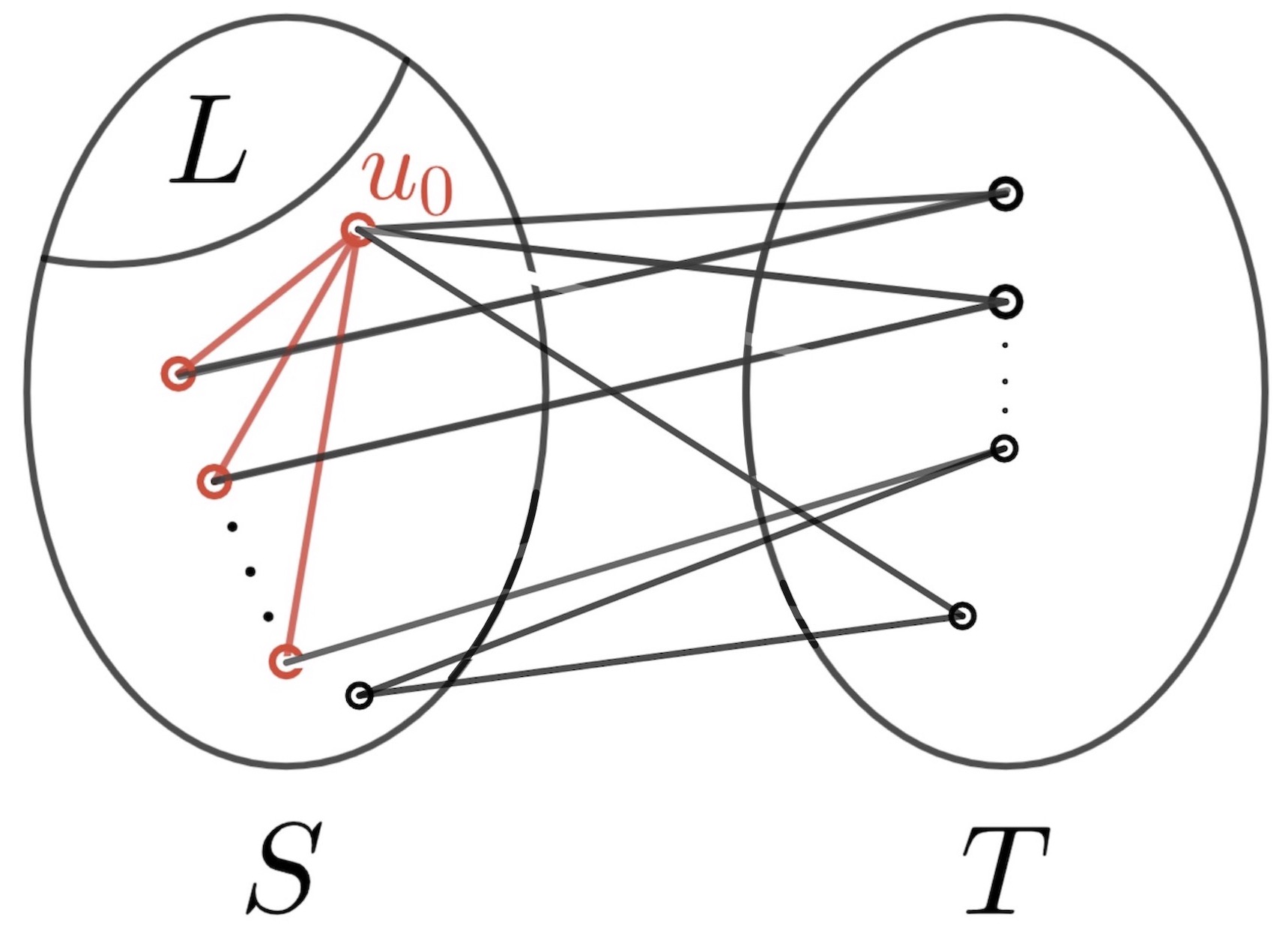}  \quad 
\includegraphics[scale=0.23]{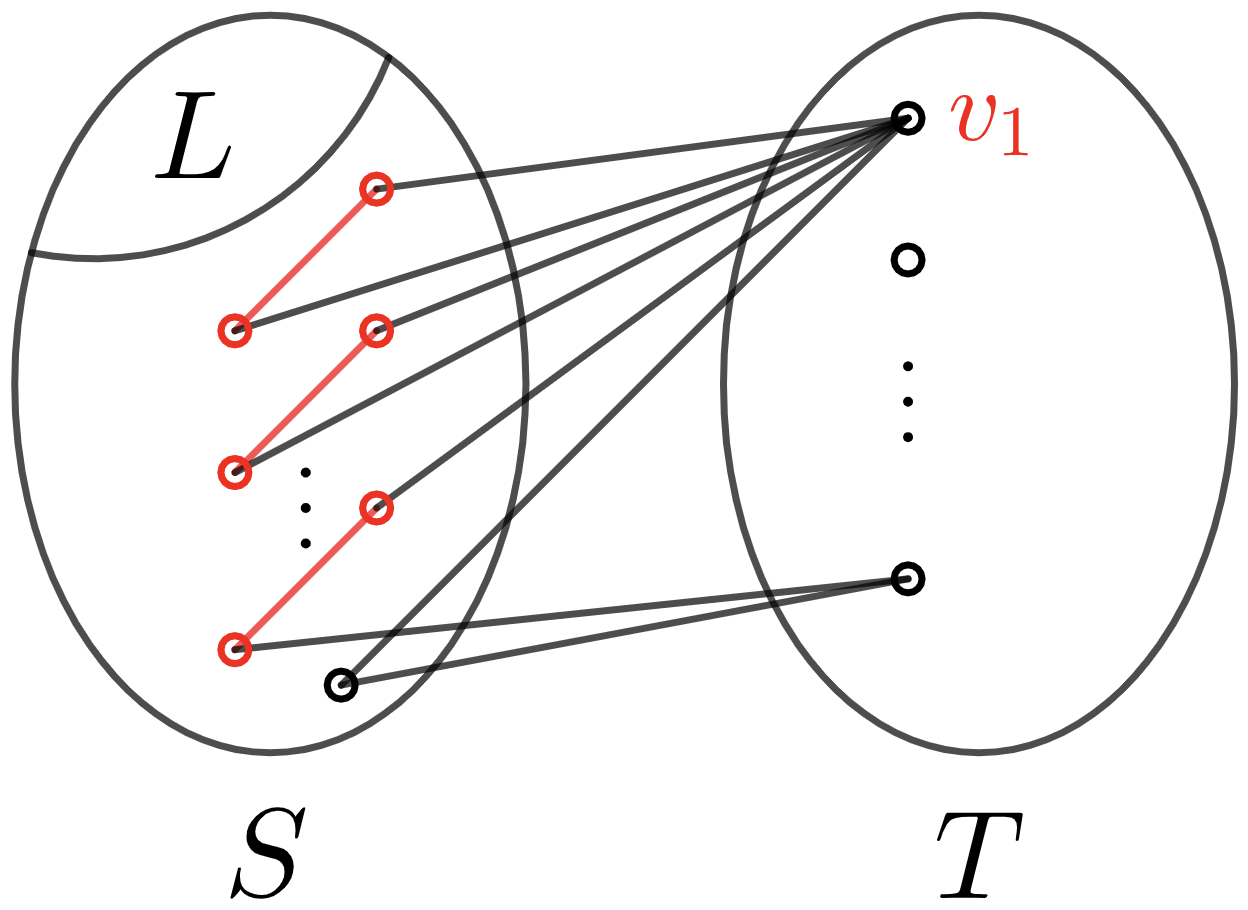}  
\end{center} 
Hence both the maximum degree and the maximum matching number of $G[S\setminus L]$  are at most  $s+k-1$, respectively. By Lemma  \ref{Chvatal76}, 
$$e(G[S\setminus L])\le f(s+k-1, s+k-1).$$
 The same argument gives
 $$e(G[T\setminus L])\le f(s+k-1, s+k-1).$$
 Since $G[S \setminus L]$ has at most $f(s+k-1, s+k-1)$ edges, then  the subgraph obtained from $G[S \setminus L]$ by deleting one vertex of each edge in $G[S\setminus L]$ contains no edges, which is an independent set of $G[S\setminus L]$. 
 By Lemma \ref{Lupper}, there exists an independent set
 $I_S\subseteq S$ such that
  \begin{align*}
  |I_S| &\ge |S\setminus L|-f(s+k-1, s+k-1) \\
  & \ge |S|- 
  16c^2(s+k)^2- (s+k)^2 \ge |S|-20c^2(s+k)^2. 
  \end{align*}
 The same argument gives that there  is an independent set $I_T\subseteq T$ with
 $$ |I_T|\ge |T|- 20c^2 (s+k)^2.
 $$
This completes the proof.
 \end{proof}

In Lemma \ref{lem336}, we have showed that
  there are two large independent sets with 
  $(\frac{1}{2} -o(1))n$ vertices, 
one in  $S$ and the other in $T$. Invoking this fact, we next shall prove that 
  $L$ is actually an empty set.

\begin{lemma} \label{Lempty}
$L$ is empty, and both $G[S]$ and $G[T]$ are $K_{1,s+k}$-free and $M_{s+k}$-free.
\end{lemma}

\begin{proof} 
 Recall that $A \mathbf{x} =\lambda_1 \mathbf{x}$ and 
 $z$ is defined as a vertex with maximum eigenvector entry and satisfies $\mathbf{x}_z=1$. So we have 
 $$d(z)\ge \sum_{w\sim z}\mathbf{x}_w=\lambda_1 \mathbf{x}_z=\lambda_1\ge   \frac{n}{2}.$$
 Hence $z\notin L$. 
  Without loss of generality, we may assume that $z\in S$.
Since the maximum degree in the induced subgraph $G[S\setminus L]$  is at most $s+k-1$ (containing no $K_{1,s+k}$), from Lemma~\ref{Lupper}, we have $|L|\le 16c^2(s+k)^2$ and
 $$
 d_S(z)=d_{S\cap L}(z)+d_{S\setminus L}(z)\le 
16c^2(s+k)^2 + s+k-1 \le 20c^2 (s+k)^2.
 $$
 Therefore,  by Lemma \ref{lem336}, we have
 \begin{eqnarray*}
 \lambda_1
 &=&\lambda_1\mathbf{x}_z=\sum_{v\sim z}\mathbf{x}_v 
 = \sum_{v\sim z, v\in S}\mathbf{x}_v+\sum_{v\sim z , v\in T}\mathbf{x}_v\\
 &=& \sum_{v\sim z, v\in S}\mathbf{x}_v+  \sum_{v\sim z, v\in I_T}\mathbf{x}_v+\sum_{v\sim z , v\in T\setminus I_T}\mathbf{x}_v\\
 &\le&  d_S(z)+\sum_{ v\in I_T}\mathbf{x}_v+\sum_{ v\in T\setminus I_T}1\\
 &\le & 20c^2 (s+k)^2 +\sum_{ v\in I_T}\mathbf{x}_v+|T|-|I_T|\\
&\le & \sum_{ v\in I_T}\mathbf{x}_v  +40 c^2(s+k)^2.
\end{eqnarray*}
 Combining (\ref{first lower bound1}), we can get 
 \begin{equation}\label{Lempty2}
 \sum_{ v\in I_T}\mathbf{x}_v\ge 
\frac{n}{2}-40 c^2(s+k)^2.
 \end{equation}

Next we are going to prove   $L=\varnothing$.

By way of contradiction, suppose that there is a vertex  $v\in L\cap S$, 
so  $d_G(v)\le (\frac{1}{2}-\frac{1}{8c(s+k)})n$.
Consider the graph $G^+$ with vertex set $V(G)$ and edge set 
$E(G^+) = (E(G) \setminus \{ vu: u\in N_G(v)\}) \cup \{vw: w\in I_T\}$. 
In this process, the number of added edges is larger than 
that of deleted edges. 
Note that $I_T$ is an independent set and then 
adding edges connecting $v$ and vertices in $I_T$ does not create any triangles or odd cycles.
So the new graph $G^+$ is still $H_{s,t_1,\ldots ,t_k}$-free. 
Note that $\mathbf{x}$ is a vector such that 
$\lambda (G)=\frac{\mathbf{x}^TA(G)\mathbf{x}}{ 
\mathbf{x}^T\mathbf{x}}$, and the Rayleigh theorem implies 
$\lambda (G^+) \ge \frac{\mathbf{x}^TA(G^+)\mathbf{x}}{ 
\mathbf{x}^T\mathbf{x}}$. 
Furthermore, 
\begin{align*}
\lambda(G^+) - \lambda(G) &\geq \frac{\mathbf{x}^T\left(A(G^+) - A(G)\right) \mathbf{x}}{\mathbf{x^T}\mathbf{x}} = \frac{2\mathbf{x}_v}{\mathbf{x}^T\mathbf{x}}\left( \sum_{w\in I_T} \mathbf{x}_w - \sum_{uv\in E(G)} \mathbf{x}_u\right) \\
& \overset{(\ref{Lempty2})}{\geq} 
\frac{2 \mathbf{x}_v}{\mathbf{x}^T\mathbf{x}} 
\left( \frac{n}{2} - 40c^2(s+k)^2 - d_G(v)\right)\\
&\geq \frac{2 \mathbf{x}_v}{\mathbf{x}^T\mathbf{x}} \left(\frac{n}{2} - 40c^2(s+k)^2 -(\frac{1}{2}-\frac{1}{8c(s+k)} )n\right) \\
&=\frac{2 \mathbf{x}_v}{\mathbf{x}^T\mathbf{x}} \left( \frac{n}{8c(s+k)}-40c^2(s+k)^2\right)> 0,
\end{align*}
where the last inequality holds for $n$ large enough 
and  $\mathbf{x}_v>0$, which follows from 
the Perron--Frobenius theorem and the fact that 
$G$ is connected. This contradicts $G$ has the largest spectral radius over all $H_{s,t_1,\ldots ,t_k}$-free graphs, so $L$ must be empty. 
Furthermore, 
the claim in the proof of Lemma \ref{lem336} 
implies that both $G[S]$ and $G[T]$ are $K_{1,s+k}$-free and $M_{s+k}$-free. 
\end{proof}

Let $G$ be an $H_{s,t_1,\ldots ,t_k}$-free graph on $n$ vertices 
with maximum spectral 
radius. 
In the previous lemmas, 
we have proved that $G$ contains at most 
$O(n^2)$ triangles and 
has $\frac{n^2}{4}- O(n)$ edges. In addition, 
 $G$ contains a bipartite subgraph with parts 
$S$ and $T$ such that 
$\frac{n}{2}-o(n) \le  |S|,|T| \le \frac{n}{2} +o(n)$. 
Next we shall refine the structure of $G$. 
We shall show that the number of triangles in 
$G$ is at most $O(n)$ and 
the number of edges in $G$ is at least 
$\frac{n^2}{4} - O(1)$, and the two vertex parts  $S,T$ 
satisfies 
$\frac{n}{2} -O(1)\le |S| ,|T| \le \frac{n}{2} + O(1)$. 
More precisely, we state these results as in the following lemma.

\begin{lemma}\label{STlambdarefine}
For $n$ and $k$ defined as before, we have
\begin{equation*}
\label{STrefine2}
e(G)\ge \frac{n^2}{4}-12(s+k)^2,
\end{equation*}
\begin{equation*} \label{est}
e(S,T)
\ge \frac{n^2}{4} -14(s+k)^2. 
\end{equation*}
\begin{equation*}
\label{STrefine1}
\frac{n}{2}-4(s+k)
\le |S|,  |T|\le \frac{n}{2}+4(s+k),
\end{equation*}
and
\begin{equation*}
\label{lambdarefine}
\frac{n}{2}-14(s+k)^2\le \delta(G)\le \lambda (G)\le \Delta (G)\le \frac{n}{2}+5(s+k).
\end{equation*}
\end{lemma}
\begin{proof}
 From Lemma \ref{Lempty},  both $G[S]$ and $G[T]$ are 
 $K_{1,s+k}$-free and $M_{s+k}$-free. 
 By Lemma \ref{Chvatal76}, so we have $e(S) + e(T) \leq 2f(s+k-1,s+k-1) < 2(s+k)^2$.
   This means that the number of triangles in $G$ is bounded above by $2(s+k)^2n$ since any triangle contains 
   an  edge of  $E(S)\cup E(T)$.
       By Lemma \ref{lemlb},
       we have $$e(G) \geq \lambda_1^2-\frac{6t}{n}
       \ge \frac{n^2}{4}  -{12(s+k)^2}.$$
Since $e(S) +e(T) \le 2(s+k)^2$, then we have 
\[  e(S,T)= e(G) -e(S) -e(T) 
\ge \frac{n^2}{4} -14(s+k)^2. \] 
Suppose that $|S|\le \frac{n}{2}-4(s+k),$ 
then $|T|=n-|S|\ge \frac{n}{2}+4(s+k)$.
       Hence
\[ e(S,T)\le |S||T|\le  \left (\frac{n}{2}-4(s+k) \right)
 \left(\frac{n}{2}+4(s+k) \right)= \frac{n^2}{4}-16(s+k)^2,\] 
        which contradicts to $e(S,T)\geq  \frac{n^2}{4}-14(s+k)^2$. 
        So we have
        $$
        \frac{n}{2}-4(s+k)
\le |S|, |T|\le \frac{n}{2}+4(s+k).
$$
Moreover, by Lemma \ref{Lempty}, 
the maximum degree of $G[S]$ and
 $G[L]$   is at most $s+k-1$, which  yields 
 $$\Delta(G)\le (\frac{n}{2}+4(s+k))+(s+k-1)< \frac{n}{2}+5(s+k).$$
  So $$\lambda_1\le \Delta(G)<  \frac{n}{2}+5(s+k).$$
Furthermore, we claim that the minimum degree of $G$ is at least $\frac{n}{2} - 14(s+k)^2$.  Otherwise,
 removing a vertex  $v$ of minimum  degree $d(v)$, we have
\begin{eqnarray*}
 e(G-v)&=&e(G)-d(v)\\
 &\ge & \frac{n^2}{4}-12(s+k)^2- \left( \frac{n}{2}-14(s+k)^2 \right)\\
 &= & \frac{n^2}{4}-\frac{n}{2}+2(s+k)^2 \\
 &> & \frac{(n-1)^2}{4}+(s+k-1)^2,
\end{eqnarray*}
which implies the induced subgraph $G-v$ contains 
a copy of $H_{s,t_1,\ldots ,t_k}$ by Theorem  \ref{thmHY}.
\end{proof}

\begin{lemma}\label{second lower bound evector2}
For all $u\in V(G)$,  we have that $\mathbf{x}_u\geq 1-\frac{120 (s+k)^2}{n}$.
\end{lemma}

\begin{proof} Without loss of generality, we may  assume that $z\in S$.  We consider the following two cases.

{\bf Step 1.}   We first consider the case  $u\in S$. 
Since  $G[S]$ is $K_{1,s+k}$-free, 
then   $d_S(u)\le s+k-1 $. 
By  Lemma \ref{STlambdarefine}, we have  
\begin{eqnarray*}
|N_T(u)|&=&d_T(u)= d(u)-d_S(u)\ge \delta(G)-d_S(u) \\
&\ge & \frac{n}{2}-14(s+k)^2-(s+k-1) \\
&\ge  &\frac{n}{2}-15(s+k)^2.
\end{eqnarray*}
Similarly, we also have $|N_T(z)| \ge \frac{n}{2}-15(s+k)^2$. 
Then 
\begin{eqnarray*}
|N_T(u)\cap N_T(z)| &=& |N_T(u)|+|N_T(z)|-|N_T(u)\cup N_T(z)| \\
&\ge& 2 \left(\frac{n}{2}-15 (s+k)^2 \right)- 
\left(\frac{n}{2}+4(s+k) \right) \\ 
&\ge & \frac{n}{2}-34 (s+k)^2.
\end{eqnarray*} 
Note that $d_T(z)\le |T|$. 
By Lemma \ref{STlambdarefine} again, we can get 
\[ d_T(z)-|N_T(u)\cap N_T(z)| \le \frac{n}{2} + 4(s+k) - 
(\frac{n}{2} - 34(s+k)^2) \le 38 (s+k)^2.\] 
Hence, we have 
\begin{eqnarray*}
\lambda_1\mathbf{x}_u-\lambda_1\mathbf{x}_z 
&=& \sum_{v\sim u} \mathbf{x}_v  - \sum_{v\sim z} \mathbf{x}_z \\
&=& \sum_{v\sim u, v\in T, v\not\sim z}\mathbf{x}_v+\sum_{v\sim u, v\in S}\mathbf{x}_v 
 -\sum_{v\sim z, v\in T, v\not\sim u}\mathbf{x}_v-\sum_{v\sim z, v\in S}\mathbf{x}_v\\
&\ge & -\sum_{v\sim z, v\in T, v\not\sim u}\mathbf{x}_v-\sum_{v\sim z, v\in S}\mathbf{x}_v\\
&\ge & -\sum_{v\sim z, v\in T,  v\not\sim u}1-\sum_{v\sim z, v\in S}1\\
&\ge & -\Bigl( d_T(z)-|N_T(u)\cap N_T(z)| \Bigr)- d_S(z)\\
&\ge & - 38(s+k)^2- (s+k)^2\\
&= & -39 (s+k)^2.
\end{eqnarray*}
Recall that $\mathbf{x}_z=1$. Therefore,  for any $u\in S$, we have
\begin{equation}\label{verct1}
\mathbf{x}_u\ge 1-\frac{39(s+k)^2}{\lambda_1}> 1-\frac{39(s+k)^2}{{n}/{2}}=1-\frac{78(s+k)^2}{n}.
\end{equation}

{\bf Step 2.}   
Now we consider the case $u\in T$. By (\ref{verct1}), we get 
 $$\lambda_1\mathbf{x}_u=\sum_{v\sim u}\mathbf{x}_v\ge \sum_{v\sim u, v\in S}\mathbf{x}_v\ge \left(1-\frac{78(s+k)^2}{n} \right)d_S(u).$$ 
By  Lemma \ref{STlambdarefine}, we can see that 
$d(u)\ge \delta (G)\ge \frac{n}{2} -14 (s+k)^2$. 
Recall that $G[T]$  is $K_{1,s+k}$-free, 
so we have $d_T(u)\le s+k-1$. 
 Then   
 \[ d_S(u)=d(u)- d_T(u)\ge 
  \frac{n}{2}-15(s+k)^2. \]
 Hence
 \begin{eqnarray*}
 \mathbf{x}_u&\ge & \frac{(1-\frac{78(s+k)^2}{n})d_S(u)}{\lambda_1} \ge \frac{(1-\frac{78 (s+k)^2}{n})(\frac{n}{2}-15(s+k)^2)}{
 \frac{n}{2}+5(s+k)} \\
& =& \frac{\frac{n}{2}-54(s+k)^2+\frac{1170(s+k)^4}{n}}{\frac{n}{2}+5(s+k)} \\
 &>&1-\frac{120(s+k)^2}{n}.
 \end{eqnarray*}

 From the above two cases, the result follows.
\end{proof}

Using this refined bound on the eigenvector entries, we will show that the partition $V=S\cup T$ is balanced (Lemma \ref{cut balanced}). 
First of all, we fix some notation for convenience. 
Let $B=K_{s, t}$ be the complete bipartite graph with partite sets $S$ and $T$, and let $G_1 = G[S] \cup G[T]$ and $G_2$ be the graph on $V(G)$ with the missing edges between $S$ and $T$, that is,  $E(G_2)=E(B) \setminus E(G)$. Note that $e(G) = 
e(G_1) + e(B)  - e(G_2)$.

From Lemma \ref{Lempty}, we know that 
both $G[S]$ and $G[T]$ are $K_{1,s+k}$-free and $M_{s+k}$-free, 
then $e(G_1)=e(S) +e(T) \le 2f(s+k-1,s+k-1) \le 2(s+k)^2 $. 
Next we shall give an improvement in the sense that 
$e(G_2)$ is close to zero. 

\begin{lemma}  \label{lemma311}
Let $G_1,G_2$ and $B$ be graphs as defined  above. Then 
\[ e(G_1) - e(G_2) 
\leq (s+k-1)^2. \] 
\end{lemma}

\begin{proof}
Without loss of generality, we may assume that $|T| \geq |S|$ 
and denote 
\begin{align*}
S' &:= \{v\in S: N(v) \subseteq T\},\\
T' &:= \{v\in T: N(v) \subseteq S\}.
\end{align*}
Since $e(G[S]) \le f(s+k-1,s+k-1)\le (s+k)^2$ by Lemma \ref{Lempty}, 
there exist at most $2(s+k)^2$ vertices in $S$   having a neighbor in $S$. Hence
$$
|S'|\ge |S|-2(s+k)^2.
$$
Similarly,
$$|T'|\ge |T|-2(s+k)^2.
$$
 Let $C\subseteq T'$ be a set having $|T|-|S|$ vertices, which 
is well-defined, as we can see from Lemma \ref{STlambdarefine} that  
$|T|-|S|\le 8(s+k)$ and 
$|T'|\ge |T|-2(s+k)^2\geq  \frac{n}{2} -4(s+k) -2(s+k)^2>8(s+k)$.
 Then $G\setminus C$ is a graph on $2|S|$ vertices such that
$$e(G)-e(C,S)=e(G\setminus C)\le \mathrm{ex} (2|S|, H_{s,t_1,\ldots ,t_k})\le\frac{(2|S|)^2}{4}+(s+k-1)^2.$$
Hence
$$e(G)\le |S|^2+|C||S|+(s+k-1)^2=|S||T|+(s+k-1)^2.$$ 
Note that $e(G_1) - e(G_2) =e(G)-e(B)$. 
This completes the proof. 
\end{proof}

\begin{lemma} \label{lemma312}
\[  \frac{2}{n} \left \lfloor\frac{n^2}{4} \right \rfloor 
 - \sqrt{|S||T|}   \le \frac{7200 (s+k)^4}{n(n-240 (s+k)^2)}. \]
\end{lemma}

\begin{proof} 
By Lemma \ref{second lower bound evector2} we have,
\begin{equation} \label{eqxx}
\mathbf{x}^T\mathbf{x} \geq n \left(1-\frac{120(s+k)^2}{n} \right)^2 
> n \left(1-\frac{240(s+k)^2}{n} \right)=n-240(s+k)^2,
\end{equation}
 and that $\lambda (B) = \sqrt{|S||T|}$.  By Lemma~\ref{STlambdarefine}, we know that 
 $e(G_1) \leq 2(s+k)^2$, we obtain
 $$e(S, T)=e(G)-e(G_1)\ge \frac{n^2}{4}- 
 12(s+k)^2-2(s+k)^2=\frac{n^2}{4}-14(s+k)^2,$$
 which implies that
\[ e(G_2)=e(B)-e(S,T)\le |S||T|- \left(\frac{n^2}{4}-14(s+k)^2 \right)\le 14(s+k)^2. \] 
Applying Lemma \ref{second lower bound evector2} again, 
   we can obtain 
  \begin{align*}
  \mathbf{x}^T A(G_2)\mathbf{x} 
  & =2\sum_{uv\in E(G_2)} \mathbf{x}_u \mathbf{x}_v\geq 2e(G_2)
   \left(1-\frac{120 (s+k)^2}{n} \right)^2  \\ 
&   \geq 2e(G_2)\left(1-\frac{240 (s+k)^2}{n}\right). 
   \end{align*}  
Combining this result together with  (\ref{first lower bound1}) 
and Lemma \ref{lemma311}, we can get 
\begin{align*}
\frac{2}{n} \left \lfloor\frac{n^2}{4} \right \rfloor
 + \frac{ 2 (s+k-1)^2}{n}
 &\overset{(\ref{first lower bound1})}{\leq} 
 \lambda (G)
 = \frac{\mathbf{x}^T (A(B) + A(G_1) - A(G_2))\mathbf{x}}{\mathbf{x}^T\mathbf{x}} \\
  &=  \frac{\mathbf{x}^T A(B)\mathbf{x}}{\mathbf{x}^T\mathbf{x}}+ \frac{\mathbf{x}^T A(G_1)\mathbf{x}}{\mathbf{x}^T\mathbf{x}}-\frac{\mathbf{x}^T A(G_2)\mathbf{x}}{\mathbf{x}^T\mathbf{x}} \\
  &\le \lambda(B) + \frac{2e(G_1)}{\mathbf{x}^T\mathbf{x}}  -\frac{2e(G_2)(1-\frac{240 (s+k)^2}{n})}{\mathbf{x}^T\mathbf{x}}\\
   & \leq \lambda(B) + \frac{2(e(G_1) -e(G_2))}{\mathbf{x}^T\mathbf{x}}+ \frac{2e(G_2)\frac{240(s+k)^2}{n}}{\mathbf{x}^T\mathbf{x}}\\
  & 
  \overset{\text{Lemma \ref{lemma311}}}{\leq} 
   \sqrt{|S||T|} + \frac{2 (s+k-1)^2}{\mathbf{x}^T\mathbf{x}} + \frac{{2\cdot 14(s+k)^2}\frac{240(s+k)^2}{n}} {\mathbf{x}^T\mathbf{x}}.
    \end{align*}
 Then we have 
 \begin{align*}  
 \frac{2}{n} \left \lfloor\frac{n^2}{4} \right \rfloor 
  - \sqrt{|S||T|}  
 & \leq   
 2(s+k-1)^2  \left(\frac{1}{\mathbf{x}^T\mathbf{x}}-\frac{1}{n} \right) 
 +  \frac{{28(s+k)^2}\frac{240(s+k)^2}{n}} {\mathbf{x}^T \mathbf{x}}\\
   & \overset{(\ref{eqxx})}{\leq}
    2(s+k)^2  \left(\frac{1}{n- {240(s+k)^2} }-\frac{1}{n}  \right)+  \frac{{6720(s+k)^4}} {n(n- {240(s+k)^2})} \\
     & = \frac{480(s+k)^4}{n(n-240(s+k)^2)} 
     +\frac{6720(s+k)^4}{n(n-240 (s+k)^2)}  \\
     & = \frac{7200 (s+k)^4}{n(n-240 (s+k)^2)}. 
 \end{align*}
 This completes the proof. 
 \end{proof}
 
 \begin{lemma}\label{cut balanced}
The sets $S$ and $T$ have sizes as equal as possible. That is
\[ \big||S| - |T|\big| \leq 1. \]
\end{lemma}

 \begin{proof}
  We assume on the contrary that $|T|\ge |S|+2$.
 We consider two cases.

 {\bf Case 1:} $n$ is even. Since $|S|+|T|=n$, we  have
 \begin{eqnarray*}
  \frac{2}{n} \left \lfloor\frac{n^2}{4} \right \rfloor  
    -\sqrt{|S||T|} & \ge&  \frac{n}{2}-\sqrt{ \left (\frac{n}{2}-1 \right) \left(\frac{n}{2}+1 \right)}\\
&=&\frac{n}{2}-\sqrt{\frac{n^2}{4}-1}=\frac{1}{\frac{n}{2}+\sqrt{\frac{n^2}{4}-1}}>\frac{1}{n}.
  \end{eqnarray*}
 So by Lemma \ref{lemma312}, we have
 $$\frac{1}{n}< \frac{2}{n}\left \lfloor\frac{n^2}{4} \right \rfloor  
  -\sqrt{|S||T|}\le  \frac{7200 (s+k)^4}{n(n-240 (s+k)^2)} .
 $$
 This is a contradiction for sufficiently  large $n$.

 {\bf Case 2:} $n$ is odd. Since $|S|+|T|=n$, we  have
   \begin{eqnarray*}
    \frac{2}{n}\left \lfloor\frac{n^2}{4} \right \rfloor  
     -\sqrt{|S||T|}
    &\ge & \frac{n^2-1}{2n}-\sqrt{\left (\frac{n-3}{2} \right) \left(\frac{n+3}{2} \right)}\\
&=& \frac{1}{2}
\left( {n-\frac{1}{n}}-{\sqrt{n^2-9}}\right)=\frac{(n-\frac{1}{n})^2-(n^2-9)}{2(n-\frac{1}{n}+\sqrt{n^2-9})}\\
&=& \frac{7+\frac{1}{n^2}}{2(n-\frac{1}{n}+\sqrt{n^2-9})}> \frac{1}{n}.
  \end{eqnarray*}
 So by Lemma \ref{lemma312} again, we get 
 $$\frac{1}{n}< \frac{2}{n} \left \lfloor\frac{n^2}{4} \right \rfloor   
  -\sqrt{|S||T|}\le\frac{7200 (s+k)^4}{n(n-240 (s+k)^2)} .
 $$
This is a contradiction for sufficiently large $n$.
Therefore
for $n$ large enough we must have that $||S| - |T|| \leq 1$.
\end{proof}

Recall that $G$ is an $H_{s,t_1,\ldots ,t_k}$-free graph with the maximum spectral 
radius. 
Finally, we will show that $e(G) = \mathrm{ex}(n, H_{s,t_1,\ldots ,t_k})$. 
In other words, $G$ also attains the maximum number of 
edges among all $H_{s,t_1,\ldots ,t_k}$-free graphs. 

\medskip 

{\bf  Proof of Theorem \ref{thmmain}}.
 By way of contradiction, we may assume that $e(G) \leq \mathrm{ex}(n, H_{s,t_1,\ldots ,t_k}) - 1$. 
By Lemma \ref{cut balanced}, we know 
that $\bigl| |S|-|T| \bigr|\le 1$. 
 Let $H$ be an $H_{s,t_1,\ldots ,t_k}$-free graph  with $\mathrm{ex}(n, H_{s,t_1,\ldots ,t_k})$ edges on the same vertex set as $G$ such that 
 the crossing edges between $S$ and $T$ span a complete bipartite graph in $H$, this is possible because every graph in $\mathrm{Ex}(n, H_{s,t_1,\ldots ,t_k})$ has a maximum cut of size $\lfloor n^2/4\rfloor$ by  Theorem \ref{thmHY}.  
We denote $E_+=E(H)\setminus E(G)$ and $E_-=E(G)\setminus E(H)$. 
 Note that $E_+$ and $E_-$ are sets of edges such that $(E(G) \cup E_+) \setminus E_- = E(H)$. Thus 
 $e(G)+|E_+| - |E_-|=e(H)$, which together with 
 $e(H)\ge  e(G) +1$ implies that
    \[   |E_+| \geq |E_-| + 1.\] 
    Furthermore, we have that $|E_-| \leq e(G[S]) + e(G[T]) < 
    2(s+k)^2$. By  Lemma \ref{STlambdarefine}, 
    we have that $|E_+| \le \lfloor \frac{n^2}{4} 
    \rfloor- e(S,T) + 2f(s+k-1,s+k-1)\le 16(s+k)^2$. 
    Now, by  Lemma~\ref{second lower bound evector2}, we have that
\begin{align*}
\lambda (H) &\geq \frac{\mathbf{x}^T A(H) \mathbf{x}}{\mathbf{x}^T\mathbf{x}} = \lambda (G) +\frac{2}{\mathbf{x}^T\mathbf{x}} \sum_{ij\in E_+} \mathbf{x}_i\mathbf{x}_j - \frac{2}{\mathbf{x}^T\mathbf{x}} \sum_{ij\in E_-} \mathbf{x}_i\mathbf{x}_j \\
& \overset{\text{Lemma}~\ref{second lower bound evector2}}{\geq}  
\lambda (G) + \frac{2}{\mathbf{x}^T\mathbf{x}} 
\left( |E_+| \Bigl(1 - \frac{120(s+k)^2}{n}\Bigr)^2 - |E_-|\right)\\
& \geq \lambda (G) + \frac{2}{\mathbf{x}^T\mathbf{x}} \left( |E_+|  - |E_-|-\frac{240(s+k)^2}{n}  |E_+| +\frac{(120(s+k)^2)^2}{n^2}  |E_+| \right)\\
& \geq \lambda (G) + \frac{2}{\mathbf{x}^T\mathbf{x}} \left( 1-\frac{240(s+k)^2}{n}  |E_+| +\frac{(120(s+k)^2)^2}{n^2}  |E_+| \right)\\
&>\lambda (G) 
\end{align*}
for sufficiently  large $n$,
where the last inequality follows by  $|E_+| < 16(s+k)^2$.
Therefore we have that for $n$ large enough, 
$\lambda (H) > \lambda (G)$, a contradiction.
Hence $e(G)=e(H)$.
By Theorem \ref{thmHY}, we know that 
$G\in \mathrm{Ex}(n,H_{s,t_1,\ldots ,t_k})$. 
The proof of Theorem \ref{thmmain} 
is complete. $\hfill\square$

\section{Concluding remarks}

To avoid unnecessary calculations, 
we did not attempt to get the best bound 
on the order of graphs in the proof. 
Our proof used the Triangle Removal Lemma, 
which means that the condition ``sufficiently large $n$'' 
is needed in our proof. It is  interesting 
to determine how large $n$ needs to be for our result.

Recently, 
Cioab\u{a}, Desai and Tait \cite{CDT21}  
investigated the largest spectral radius of
 an $n$-vertex graph that does not contain the odd-wheel 
 graph $W_{2k+1}$, 
which is the graph obtained by joining a vertex to a cycle 
of length $2k$. Moreover, they raised 
the following more general conjecture. 

\begin{conjecture} \label{conj}
Let $F$ be any graph such that the graphs in $\mathrm{Ex}(n,F)$ 
are Tur\'{a}n graphs plus $O(1)$ edges. 
Then for sufficiently large $n$, 
a graph attaining the maximum spectral radius 
among all $F$-free graphs is a member of $\mathrm{Ex}(n,F)$. 
\end{conjecture}

We say that  $F$ is  edge-color-critical if  
there exists an edge $e$ of $F$ such that 
$\chi (F-e) < \chi (F)$. 
Let $F$ be an edge-color-critical graph with $\chi (F)=r+1$. 
By a result of Simonovits \cite{Sim66}
and a result of Nikiforov \cite{Niki09}, 
we know that $\mathrm{Ex}(n,F)=\mathrm{Ex}_{sp}(n,F) =
\{T_r(n)\}$ for sufficiently large $n$. This  
 shows that Conjecture \ref{conj}  is true for 
all edge-color-critical graphs.  
As we mentioned before, Theorem \ref{thmCFTZ20} 
says that Conjecture \ref{conj} 
holds for the $k$-fan graph $F_k$. 
In addition, our main result (Theorem \ref{thmmain}) 
tells us that    Conjecture \ref{conj} also holds 
for the flower graph $H_{s,t_1,\ldots ,t_k}$. 
Note that both $F_k$ and $H_{s,t_1,\ldots ,t_k}$ are not edge-color-critical. 

Let  $S_{n,k}$ be the graph consisting of a clique on $k$ vertices and an independent set on $n-k$ vertices in which each vertex of the clique is adjacent to each vertex of the independent set. 
Clearly, we can see that 
$S_{n,k}$  does not contain $F_k$ as a subgraph. 
Recently, Zhao, Huang and Guo \cite{ZHG21} 
proved that $S_{n,k}$ is 
the unique graph attaining the maximum signless Laplacian spectral radius among all graphs of order $n$ containing no $F_k$ 
for $n\ge 3k^2-k-2$. 
So it is a natural question 
to consider the maximum signless Laplacian spectral radius 
among all graphs containing no $C_{k,q}$, 
the graph defined as $k$  cycles of odd-length $q$ intersecting in a common vertex. 
We write $q(G)$ for the 
  signless Laplacian spectral radius, i.e., 
 the largest eigenvalue of 
 the {\it signless Laplacian matrix} $Q(G)=D(G) + 
 A(G)$, where $D(G)=\mathrm{diag}(d_1,\ldots ,d_n )$ 
 is the degree diagonal matrix and 
 $A(G)$ is the adjacency matrix. 
We end with the following conjecture 
(Clearly, when $t=1$, 
our conjecture reduces to the result of Zhao et al. \cite{ZHG21}). 

\begin{conjecture}
For integers $k\ge 2, t\ge 1$ and $q=2t+1$, 
there exists an integer $n_0(k,t)$ such that 
if $n\ge n_0(k,t)$ and 
$G$ is a $C_{k,q}$-free graph on $n$ 
vertices, then 
\[  q(G) \le q(S_{n,kt}), \]
equality holds if and only if $G=S_{n,kt}$. 
\end{conjecture} 

{\bf Remark.} 
After we submitted our paper, 
this conjecture was recently solved by Chen, Liu and Zhang \cite{CLZ2021}. 

Another interesting problem 
on this topic is to 
determine the Tur\'{a}n number  
of $C_{k,q}$ for even $q$. 
In general, it is  challenging 
to determine the  Tur\'{a}n number 
of $H_{s,t_1,\ldots ,t_k}$ where the cycles have even lengths.

\subsection*{Acknowledgements}
The authors would like to thank anonymous reviewers for their valuable comments and suggestions to improve the presentation of the paper. 
The first author would like to express his sincere thanks to 
Prof. Lihua Feng 
and Lu Lu for many illuminating discussions. 
This work was supported by  NSFC (Grant No. 11931002).

\end{document}